\documentclass[12pt]{amsart}
\usepackage{amssymb,amsmath,amsfonts,latexsym}
\usepackage{bm}
\newcommand{\newsection}[1]{\setcounter{equation}{0} \section{#1}}
\newcommand{\bea}{\begin{eqnarray}}
\newcommand{\eea}{\end{eqnarray}}

\newcommand{\cla}{\mathcal{A}}
\newcommand{\clb}{\mathcal{B}}

\newcommand{\cle}{\mathcal{E}}
\newcommand{\clf}{\mathcal{F}}
\newcommand{\clg}{\mathcal{G}}
\newcommand{\clh}{\mathcal{H}}

\newcommand{\cll}{\mathcal{L}}
\newcommand{\clm}{\mathcal{M}}

\newcommand{\clo}{\mathcal{O}}

\newcommand{\cls}{\mathcal{S}}

\newcommand{\clw}{\mathcal{W}}

\newcommand{\z}{\bm{z}}
\newcommand{\w}{\bm{w}}

\newcommand{\raro}{\rightarrow}

\def \qed {\hfill \vrule height6pt width 6pt depth 0pt}
\def\textmatrix#1&#2\\#3&#4\\{\bigl({#1 \atop #3}\ {#2 \atop #4}\bigr)}
\def\dispmatrix#1&#2\\#3&#4\\{\left({#1 \atop #3}\ {#2 \atop #4}\right)}
\newcommand{\be}{\begin{equation}}
\newcommand{\ee}{\end{equation}}
\newcommand{\ben}{\begin{eqnarray*}}
\newcommand{\een}{\end{eqnarray*}}

\newcommand{\NI}{\noindent}

\newcommand{\bi}{\begin{itemize}}
\newcommand{\ei}{\end{itemize}}

\newtheorem{Theorem}{\sc Theorem}[section]

\newtheorem{Proposition}[Theorem]{\sc Proposition}
\newtheorem{Corollary}[Theorem]{\sc Corollary}
\newtheorem{Definition}[Theorem]{\sc Definition}
\newtheorem{Example}[Theorem]{\sc Example}
\newtheorem{Remark}[Theorem]{\sc Remark}

\theoremstyle{plain}

\theoremstyle{definition}

\numberwithin{equation}{section}

\begin{document}

\title[Dilations, Wandering subspaces, and inner functions]{Dilations, Wandering subspaces, and inner functions}

\author[Bhattacharjee]{M. Bhattacharjee}
\address{Indian Institute of Science, Department of Mathematics, Bangalore, 560012, India}
\email{monojit12@math.iisc.ernet.in}

\author[Eschmeier]{J. Eschmeier}
\address{Fachrichtung Mathematik, Universit\"{a}t des Saarlandes, Postfach 151150, D-66041 Saarbr\"{u}cken, Germany}
\email{eschmei@math.uni-sb.de}

\author[Keshari]{Dinesh K. Keshari}
\address{Indian Statistical Institute, Statistics and Mathematics Unit, 8th Mile, Mysore Road, Bangalore, 560059, India}
\email{dinesh@isibang.ac.in}

\author[Sarkar]{Jaydeb Sarkar}
\address{Indian Statistical Institute, Statistics and Mathematics Unit, 8th Mile, Mysore Road, Bangalore, 560059, India}
\email{jay@isibang.ac.in, jaydeb@gmail.com}

\subjclass[2010]{Primary 47A13; Secondary 46E22, 47A15, 47A20, 47A45
47A45}


\keywords{Dilations, joint invariant subspaces, wandering subspaces,
Drury-Arveson space, reproducing kernel Hilbert spaces, multipliers,
inner functions}

\begin{abstract}
The objective of this paper is to study wandering subspaces for
commuting tuples of bounded operators on Hilbert spaces. It is shown
that, for a large class of analytic functional Hilbert spaces $\clh_K$ on the
unit ball in $\mathbb C^n$, wandering subspaces for restrictions of
the multiplication tuple $M_z = (M_{z_1}, \ldots ,M_{z_n})$ can be
described in terms of suitable $\clh_K$-inner functions. We prove that 
$\clh_K$-inner functions are contractive multipliers and deduce a result
on the multiplier norm of quasi-homogenous polynomials as an application.
Along the way we prove a refinement of a result of Arveson on the uniqueness of
minimal dilations of pure row contractions.
\end{abstract}

\maketitle









\newsection{Introduction}

Let $T = (T_1, \ldots, T_n)$ be an $n$-tuple of commuting bounded
linear operators on a complex Hilbert space $\clh$. A closed subspace $\clw \subset \clh$ is called 
a \textit{wandering subspace} for $T$ if
\[
\mathcal{W}\perp T^{\bm{k}} \mathcal{W} \quad \quad (\bm{k} \in
\mathbb{N}^n\setminus \{\bm{0}\}).
\]
We say that $\clw$ is a \textit{generating wandering subspace} for $T$ if in addition
\[
\clh = \overline{\mbox{span}}\{T^{\bm{k}} \clw : \bm{k} \in \mathbb{N}^n\}.
\]

Wandering subspaces were defined by Halmos in \cite{PH}.
One of the main observations from \cite{PH} is the following.
Let $\cle$ be a Hilbert space and let $M_z: H^2_{\cle}(\mathbb{D}) \rightarrow H^2_{\cle}(\mathbb{D})$ be
the operator of multiplication with the argument on the $\cle$-valued Hardy space $H^2_{\cle}(\mathbb{D})$
on the unit disc $\mathbb D$. Suppose that $\cls$ is a non-trivial closed $M_z$-invariant subspace of
$H^2_{\cle}(\mathbb{D})$. Then
\[
\clw = \cls \ominus z \cls
\]
is a wandering subspace for $M_z|_{\cls}$ such that
\[
M_z^p \clw \perp M_z^q \clw
\]
for all $p \neq q $ in $\mathbb{N}$ and
\[
\cls = \overline{\mbox{span}}\{ z^m \clw: m \in \mathbb{N}\}.
\]
Hence
\[
\cls = \mathop{\oplus}_{m=0}^\infty z^m\clw
\]
and up to unitary equivalence
\[
M_z|_{\cls} \;\mbox{on}\; \cls \cong M_z \; \mbox{on}\; H^2_{\clw}(\mathbb{D}).
\]
In particular, we have $\cls = V (H^2_{\clw}(\mathbb{D}))$, where $V
: H^2_{\clw}(\mathbb{D}) \raro H^2_{\cle}(\mathbb{D})$ is an
isometry and $V M_z = M_z V$.  One can show (see Lemma V.3.2 in \cite{NF} for details
and more precise references)
that any such intertwining isometry $V$ acts as the
multiplication operator $V = M_{\Theta}: H^2_{\clw}(\mathbb{D}) \rightarrow H^2_{\cle}(\mathbb{D}),$
$f \mapsto \Theta f,$ with a bounded analytic function
$\Theta \in H^\infty_{\clb(\clw, \cle)}(\mathbb{D})$ such that $\Theta$ possesses isometric
boundary values almost everywhere. In this case
\[
\cls = \Theta H^2_{\clw}(\mathbb{D})
\]
and (cf. Theorem \ref{inv-tuple} below)
\begin{equation}\label{paramet}
\cls \ominus z \cls = \Theta \clw.
\end{equation}
Thus the wandering subspaces of $M_z$
on $H^2_{\cle}(\mathbb{D})$ can be described using the
Beurling-Lax-Halmos representation of $M_z$-invariant subspaces
of $H^2_{\cle}(\mathbb{D})$.

Much later,  in Aleman, Richter and Sundberg \cite{ARS}, it was
shown that every $M_z$-invariant closed subspace of the Bergman space on the
unit disc $\mathbb D$ is
generated by a wandering subspace. More precisely, let $\cls$ be a closed
$M_z$-invariant subspace of the Bergman space $L^2_a(\mathbb{D})$.
Then
\[
\cls = [\cls \ominus z \cls],
\]
where the notation $[M]$ is used for the smallest closed $M_z$-invariant
subspace containing a given set $M \subset L^2_a(\mathbb{D})$. The above
result
of Aleman, Richter and Sundberg has been extended by Shimorin
(see \cite{SS1} and \cite{SS2}) who replaced the multiplication operator
$M_z$ on the Bergman space by left invertible Hilbert space operators 
satisfying suitable operator inequalities.


In this paper we study wandering subspaces
for commuting tuples of operators on Hilbert spaces. More precisely,
let $T = (T_1, \ldots, T_n)$ be a commuting tuple of bounded operators on a
Hilbert space $\clh$. Suppose that  $\clw = \clh \ominus \sum_{i=1}^n T_i
\clh$ is a generating wandering subspace for $T$. We are
interested in the following general question: given a closed
$T$-invariant subspace $\cls \subset \clh$, are there natural conditions which
ensure that $T|_{\cls} = (T_1|_{\cls} , \ldots , T_n|_{\cls})$ has a generating wandering subspace again?

In view of the known one-variable results it seems natural to study this problem
first in the particular case where $T$ is the tuple
$M_z = (M_{z_1}, \ldots ,M_{z_n})$ consisting of the multiplication operators with
the coordinate functions on some classical reproducing kernel Hilbert spaces such as
the Hardy space, the Bergman space or the Drury-Arveson
space on the unit ball $\mathbb{B}^n$ of $\mathbb C^n$.

The main purpose of this paper, however, is to parameterize the
wandering subspaces and, in particular, to extend the representation
(\ref{paramet}) to a large class of commuting operator tuples. The
above question concerning the existence of generating wandering
subspaces, even for classical reproducing kernel Hilbert spaces over
the unit ball in $\mathbb{C}^n$, seems to be more elusive.

Our primary motivation for studying wandering subspaces comes from
recent results on Beurling-Lax-Halmos type representations of
invariant subspaces of commuting tuples of operators (see Theorem
\ref{inv} below or \cite{BB1} and \cite{JS1}). Our study is also
motivated by Hedenmalm's theory \cite{HH} of Bergman inner functions
for shift-invariant subspaces of the Bergman space on the unit disc
$\mathbb{D}$. This concept has been further generalized by Olofsson
\cite{O2,O3} to obtain parameterizations of wandering subspaces of
shift-invariant subspaces for the weighted Bergman spaces on $\mathbb D$
corresponding to the kernels
\[
K_m(z,w) = (1 - z \bar{w})^{-m} \quad (m \in \mathbb N).
\]

Our observations heavily depend on the
existence of dilations for commuting row contractions (see Section
\ref{dilation}).
For instance, let $T = (T_1, \ldots, T_n)$ be a pure
commuting contractive tuple on a Hilbert space $\clh$.
Let $\Pi :\clh \raro H^2_n(\cle)$ be the Arveson
dilation of $T$, and let $\tilde{\Pi} : \clh \raro H^2_n(\tilde{\cle})$ be an
arbitrary dilation of $T$ (see Section \ref{dilation}). Then our main
uniqueness result, which may be of independent interest, yields
an isometry $V : \cle \raro \tilde{\cle}$ such that the following
diagram commutes (Corollary \ref{u-dil2}):
\vspace{-.5cm}

 \setlength{\unitlength}{3mm}
 \begin{center}
 \begin{picture}(40,16)(0,0)
\put(15,3){$\clh$}\put(19,1.6){$\tilde{\Pi}$}
\put(22.9,3){$H^2_n(\tilde{\cle})$} \put(22, 10){$H^2_n(\cle)$}
\put(22.8,9.2){ \vector(0,-1){5}} \put(15.8, 4.3){\vector(1,1){5.8}}
\put(16.4,
3.4){\vector(1,0){6}}\put(16.5,8){$\Pi$}\put(24.3,7){$I_{H^2_n}
\otimes V$}
\end{picture}
\end{center}

In the one-dimensional case $n = 1$, some of our observations concerning wandering subspaces
are closely related to results of Shimorin \cite{SS1,
SS2}, Ball and Bolotnikov \cite{BB1, BB2, BB3} and Olofsson
\cite{O1, O2, O3}.

In Section 2 we define the notion of a minimal dilation for pure
commuting row contractions $T$ and show that the Arveson dilation is
a minimal dilation of $T$. In Section 3 we show that minimal dilations
are uniquely determined and that each dilation of a pure commuting
row contraction factorizes through its minimal dilation. If $\cls$ is a
closed invariant subspace for $T$, then by
dualizing the minimal dilation of the restriction $T|_{\cls}$ one obtains
a representation of $\cls$ as the image of a partially isometric module
map $\Pi: H^2_n(\cle) \rightarrow H$ defined on a vector-valued Drury-Arveson
space. In Section 4 the uniqueness and factorization results for minimal
dilations are used to prove corresponding results for the representation $\Pi$.
In Section 5 we show that any representation $\Pi: H^2_n(\cle) \rightarrow H$
of a $T$-invariant subspace $\cls$ induces a unitary representation of the
associated wandering subspace $\clw = \cls \ominus \sum^n_{i=1} T_i \cls$.
Finally, in Section 6 we show that, in the particular case that 
$T \in \clb(\clh)^n$ is the the multiplication tuple $M_z =(M_{z_1}, \ldots ,M_{z_n})$ 
on a contractive analytic functional Hilbert space $H(K)$ on $\mathbb B^n$, 
the above representation of the wandering subspace $\clw = \cls \ominus \sum^n_{i=1} M_{z_i} \cls$ 
is given by a suitably defined $H(K)$-inner function. We show that $H(K)$-inner
functions are contractive multipliers and apply these results to deduce that the
norm and the multiplier norm for quasi-homogeneous polynomials on the Drury-Arveson
space coincide. We conclude with an example showing that in contrast to the one-dimensional 
case in dimension $n > 1$, even for the nicest
analytic functional Hilbert spaces on $\mathbb B^n$ such as the Hardy space, the Bergman space or
the Drury-Arveson space, there
are $M_z$-invariant subspaces which do not possess a generating wandering 
subspace.


The authors are grateful to Orr Shalit for helpful discussions.

\section{Minimal Dilations}\label{dilation}

We begin with a brief introduction to the theory of dilations for
commuting row contractions.

Let $T=(T_1,\ldots,T_n)$ be an $n$-tuple of bounded linear operators
on a complex Hilbert space $\mathcal{H}$. We denote by $P_T : \mathcal{B}(\mathcal{H})\to
\mathcal{B}(\mathcal{H})$ the completely positive map defined by
\[
P_{T}(X)=\sum_{i=1}^{n}T_iX T_i^*.
\]
An $n$-tuple $T$ is called a \textit{row contraction} or simply a
\textit{contractive tuple} if
\[
P_T(I_{\clh}) \leq I_{\clh}.
\]
If $T$ is a row contraction, then
\[
I_{\mathcal{H}}\geq
P_{T}(I_{\mathcal{H}})\geq P_{T}^2(I_{\mathcal{H}})\geq \cdots \geq
P_{T}^m(I_{\mathcal{H}})\geq \cdots \geq 0.
\]
Hence the limit
\[P_{\infty}(T) = {\rm SOT-}\lim_{m \raro \infty} P_{T}^m(I_{\mathcal{H}})
\]
exists and satisfies the inequalities $0\leq P_{\infty}(T)\leq
I_{\mathcal{H}}$. A row contraction $T$ is called \textit{pure} (cf. \cite{Ar} or \cite{Pop})
if $P_{\infty}(T)=0$.

Let $\lambda \geq 1$ and let $K_\lambda : \mathbb{B}^n \times
\mathbb{B}^n \raro \mathbb{C}$ be the positive definite function
defined by
\[
K_\lambda(\z, \w) = (1 - \sum_{i=1}^n z_i \bar{w}_i)^{-\lambda}.
\]
Then the Drury-Arveson space $H^2_n$,
the Hardy space $H^2(\mathbb{B}^n)$, the Bergman space
$L^2_a(\mathbb{B}^n)$, and the weighted Bergman spaces $L^2_{a,
\alpha}(\mathbb{B}^n)$ with $\alpha > -1$, are the reproducing kernel
Hilbert spaces with kernel $K_{\lambda}$ where $\lambda = 1, n$, $n+1$
and $n + 1 + \alpha$, respectively. The tuples $M_z = (M_{z_1}, \ldots , M_{z_n})$
of multiplication operators with the coordinate functions
on these reproducing kernel Hilbert spaces are examples of
pure commuting contractive tuples of operators.

Let $K$ be a positive definite function on
$\mathbb{B}^n$ holomorphic in the first and anti-holomorphic in the
second variable. Then the functional Hilbert space $\clh_K$
with reproducing kernel $K$ consists of analytic functions on $\mathbb B^n$. For any Hilbert space $\cle$, the
$\cle$-valued functional Hilbert space with reproducing kernel
\[
\mathbb B^n \times \mathbb B^n \raro \clb(\cle), \; (\z, \w) \mapsto K(\z,\w) I_{\cle}
\]
can canonically be identified with the tensor product Hilbert
space $\clh_K \otimes \cle$. To simplify the notation, we
often identify $H^2_n \otimes \cle$ with the $\cle$-valued
Drury-Arveson space $H^2_n(\cle)$.

Let $T$ be a commuting row contraction on $\clh$ and let $\cle$ be
an arbitrary Hilbert space. An isometry $\Gamma : \clh \raro
H^2_n(\cle)$ is called a \textit{dilation} of $T$ if
\[
M_{z_i}^* \Gamma =  \Gamma T_i^* \quad \quad (i = 1, \ldots, n).
\]
Since $M_z \in \clb(H^2_n(\cle))^n$ is a pure row contraction and since 
a compression of a pure row contraction to a co-invariant subspace remains
pure, any commuting row contraction possessing a dilation of the above
type is necessarily pure.

Let $\Gamma : \clh \raro H^2_n(\cle)$ be a dilation of $T$.
Since the $C^*$-subalgebra of $\clb(H^2_n)$
generated by $(M_{z_1}, \ldots, M_{z_n})$ has the form
\[
C^*(M_z) = \overline{\mbox{span}} \{M_z^{\bm{k}} M_z^{*\bm{l}} :
\bm{k}, \bm{l} \in \mathbb{N}^n\}
\]
(see Theorem 5.7 in \cite{Ar}), the space
\[
M = \overline{{\rm span}} \{z^{\bm{k}} \Gamma h : \bm{k} \in \mathbb{N}^n, h \in \clh\}
\]
is the smallest reducing subspace for $M_z$ on $H^2_n(\cle)$ containing the image of
$\Gamma$. As a reducing subspace for $M_z \in \clb(H^2_n(\cle))^n$ the space $M$ has the form
\[
M = \bigvee_{{\bm k} \in \mathbb N^n} z^{{\bm k}} \cll = H^2_n(\cll) \quad {\rm with} \quad \cll = M \cap \cle.
\]
We call $\Gamma$ a \textit{minimal dilation} of $T$ if
\[
H^2_n(\cle) = \overline{{\rm span}} \{z^{\bm{k}} \Gamma h : \bm{k} \in
\mathbb{N}^n, h \in \clh\}.
\]

We briefly recall a canonical way to construct minimal dilations for
pure commuting contractive tuples (cf. \cite{Ar}). Let $T$ be a pure
commuting contractive tuple on $\clh$. Define
\[
\cle_c = \overline{{\rm ran}}(I_{\clh} - P_T(I_{\clh})), \quad \quad \quad D =
(I_{\clh} - P_T(I_{\clh}))^{\frac{1}{2}}.
\]
Then the operator $\Pi_c : \clh \raro H^2_n(\cle_c)$ defined by
\[
(\Pi_c h)(\bm{z}) = D (I_{\clh} - \sum_{i=1}^n z_i T_i^*)^{-1}h \quad \quad (\bm{z} \in \mathbb{B}^n, h \in \clh).
\]
is a dilation of $T$. Let
\[
M = \bigvee_{{\bm k} \in \mathbb N^n} z^{{\bm k}} \cll = H^2_n(\cll) \quad {\rm with} \quad \cll = M \cap \cle_c
\]
be the smallest reducing subspace for $M_z \in \clb(H^2_n(\cle_c))^n$ which contains the image of $\Pi_c$ and let
$P_{\cle_c}$ be the orthogonal projection of $H^2_n(\cle_c)$ onto
the subspace consisting of all constant functions. Since
\[
P_{\cle_c} =I_{H^2_n(\cle_c)} - \sum_{i=1}^n M_{z_i} M_{z_i}^*,
\]
we obtain that
\[
Dh = P_{\cle_c}(\Pi_c h) \in H^2_n(\cll) \cap \cle_c = \cll
\]
for each $h \in \clh$. Hence $\Pi_c$ is a minimal dilation of $T$.
\vspace{.5cm}

\newsection{uniqueness of minimal dilations}

Using a refinement of an idea of Arveson \cite{Ar}, we obtain the following sharpened
uniqueness result for minimal dilations of pure commuting contractive tuples.

\begin{Theorem}\label{u-dil}
Let $T \in \clb(\clh)^n$ be a pure commuting contractive tuple on a Hilbert space $\clh$
and let $\Pi_i : \clh \raro H^2_n({\cle_i})$, $i = 1, 2$, be a pair of minimal
dilations of $T$. Then there exists a unitary operator $U \in \clb(\cle_1,
\cle_2)$ such that
\[\Pi_2 = (I_{H^2_n} \otimes U) \Pi_1.\]
\end{Theorem}
\NI\textsf{Proof.}
As an application of Theorem 8.5 in \cite{Ar} one can show that there is a
unitary operator $W: H^2_n(\cle_1) \rightarrow H^2_n(\cle_2)$ which intertwines
the tuples $M_z$ on $H^2_n(\cle_1)$ and $H^2_n(\cle_2)$.
We prefer to give a more direct proof containing some
simplifications which are possible in the pure case.

Let $\clb = C^*(M_z) \subset \clb(H^2_n)$ be the $C^*$-algebra
generated by the multiplication tuple $(M_{z_1}, \ldots,
M_{z_n})$ on the scalar-valued Drury-Arveson space $H^2_n$.
Denote by $\mathcal{A}$ the unital subalgebra of $\clb$ consisting of all
polynomials in $(M_{z_1}, \ldots, M_{z_n})$.

\NI Let $\Pi : \clh \raro H^2_n(\clf)$ be a minimal dilation of
$T$. The map $\varphi : \clb \raro \clb(\clh)$ defined by
\[
\varphi(X) = \Pi^* (X \otimes I_{\clf}) \Pi
\]
is completely positive and unital. For each polynomial
$p \in \mathbb{C}[z_1, \ldots, z_n]$ and each operator $X \in \clb$,
we have $\varphi(p(M_z)) = p(T)$ and, since ${\rm ran}(\Pi)^{\perp}$ is
invariant for $M_z$, it follows that
\[
\varphi(p(M_z) X) = \Pi^* (p(M_z)X \otimes 1_{\clf}) \Pi = \varphi(p(M_z))\varphi(X).
\]
Hence $\varphi : \clb \raro \clb(\clh)$ is an $\cla$-morphism in the sense
of Arveson (Definition 6.1 in \cite{Ar}). The minimality of $\Pi$ as a dilation
of $T$ implies that the map $\pi : \clb \raro \clb(H^2_n(\clf))$ defined by
\[
\pi(X) = X \otimes I_{\clf}
\]
is a minimal Stinespring dilation of $\varphi$.
Now let $\Pi_i : \clh \raro H^2_n({\cle_i})$, $i = 1, 2$, be a pair
of minimal dilations of $T$ and let $\varphi_i : \clb \raro \clb(\clh), \;
\pi_i : \clb \raro \clb(H^2_n(\cle_i))$ be the maps induced by $\Pi_i$ as
explained above. Then $\varphi_1 = \varphi_2$ on $\cla$, and hence a
direct application of Lemma 8.6 in \cite{Ar} shows that there is a
unitary operator $W:  H^2_n(\cle_1) \raro H^2_n(\cle_2)$  such that
$W (X \otimes I_{\cle_1}) = (X \otimes I_{\cle_2}) W$ for all $X \in \clb$
and such that
\[
W \Pi_1 = \Pi_2.
\]
Since $W$ and $W^*$ both intertwine the tuples $M_z \otimes I_{\cle_1}$
and $M_z \otimes I_{\cle_2}$, it follows as a very special case of the
functional commutant lifting theorem for the Drury-Arveson space
(see Theorem 5.1 in \cite{BTV} or Theorem 3.7 in \cite{AE}) that
$W$ and $W^*$
are induced by multipliers, say $W = M_a$  and $W^* = M_b$, where
$a: \mathbb{B}^n \raro \clb(\cle_1,\cle_2)$ and $b: \mathbb{B}^n \raro
\clb(\cle_2,\cle_1)$ are operator-valued multipliers.
Since
\[
M_{ab} = M_a M_b = I_{H^2_n(\cle_2)} \quad \mbox{and} \quad M_{ba} = M_b M_a =
I_{H^2_n(\cle_1)},
\]
it follows that $a(\bm{z}) = b(\bm{z})^{-1}$ are
invertible operators for all $\bm{z} \in \mathbb{B}^n$. Moreover, since
\[
K(\z,\w) \eta = (M_a M^*_a K(\cdot, \w) \eta) (\z) = a(\z) K(\z, \w) a(\w)^* \eta,
\]
holds for all $\z, \w \in \mathbb{B}^n$ and $\eta \in \cle_2$, it follows that
$a(\z) a(\w)^* = I_{\cle_2}$ for all $\z, \w \in \mathbb{B}^n$. Hence
the operators $a(\z)$ are unitary with
$a(\z) = a(\w)$ for all $\z,\w \in \mathbb{B}^n$. Let $U = a(\z)$ be the constant
value of the multiplier $a$. Then $U \in \clb(\cle_1,\cle_2)$ is a unitary operator with
$W = I_{H^2_n} \otimes U$. Hence the proof is complete.
\qed

\begin{Remark}\label{canonical}
Let $\Pi_c : \clh \raro H^2_n(\cle_c)$ be the Arveson dilation of a
pure commuting row contraction  $T$ as in Section \ref{dilation} and
let $\Pi : \clh \raro H^2_n(\cle)$ be an arbritrary minimal dilation of $T$.
Then there is a unitary operator $U \in \clb(\cle_c,\cle)$ such that
\[
\Pi = (I_{H^2_n} \otimes U) \Pi_c.
\]
In this sense the Arveson dilation $\Pi_c$ is the unique minimal dilation of $T$.
In the sequel we call $\Pi_c$ the {\rm canonical dilation} of $T$.
\end{Remark}

As a consequence of Theorem \ref{u-dil}, we obtain the following
factorization result (see Theorem 4.1 in \cite{JS3} for the
single-variable case).

\begin{Corollary}\label{u-dil2}\textsf{(Canonical factorizations of dilations)}
Let $T \in \clb(\clh)^n$ be a pure commuting contractive tuple and let ${\Pi} :
\clh \raro H^2_n({\cle})$ be a dilation of $T$. Then there exists an
isometry $V \in \clb(\cle_c, {\cle})$ such that \[ {\Pi} =
(I_{H^2_n} \otimes V) \Pi_c.\]
\end{Corollary}

\NI\textsf{Proof.} As shown in Section \ref{dilation} there is a closed
subspace $\clf \subseteq {\cle}$ such that
\[
H^2_n(\clf) = \overline{\mbox{span}} \{z^{\bm{k}} {\Pi} \clh : \bm{k} \in
\mathbb{N}^n\}.
\]
Then by definition ${\Pi} : \clh \raro H^2_n(\clf)$ is a
minimal dilation of $T$. By Theorem \ref{u-dil}, there exists a
unitary operator $U : \cle_c \raro \clf$ such that
\[
{\Pi} = (I_{H^2_n} \otimes U) \Pi_c.
\]
Clearly, $U$ regarded as an operator with values in $\cle$, defines an isometry $V$ with
the required properties. \qed

\newsection{Joint invariant subspaces}

In this section we study the structure of joint invariant subspaces
of pure commuting row contractions.

We begin with the following characterization of invariant subspaces from \cite{JS2}
(Theorem 3.2). The proof follows as an elementary application of the above dilation results.

\begin{Theorem}\label{inv}
Let $T = (T_1, \ldots, T_n)$ be a pure commuting contractive tuple
on $\clh$ and let $\cls$ be a closed subspace of $\clh$. Then $\cls$ is
a joint $T$-invariant subspace of $\clh$ if and only if there exists
a Hilbert space $\cle$ and a partial isometry $\Pi \in \clb(H^2_n(\cle), \clh)$ with
$\Pi M_{z_i} = T_i \Pi$ for $i = 1, \ldots ,n$ and
\[
\cls = \Pi (H^2_n(\cle)).
\]
\end{Theorem}

\NI\textsf{Proof.} We indicate the main ideas. Let $\cls \subset \clh$ be a closed invariant
subspace for $T$. Since
\[
\langle P^m_{(T| \cls)}(I_{\cls})x,x \rangle = \sum_{| \bm{k} | = m} \frac{|\bm{k} |!}{\bm{k} !}
\| P_{\cls} T^{\ast \bm{k}}x \|^2 \stackrel{m}{\rightarrow} 0
\]
for each $x \in \cls$, the restriction $T|_{\cls}$ is a pure commuting row contraction again.
Let $\Pi_{\cls} : \cls \raro H^2_n(\cle)$ be a dilation of $T|_{\cls}$ and let
$i_{\cls} : \cls \raro \clh$ be the inclusion map. Then
\[
\Pi = i_{\cls} \circ \Pi_{\cls}^* : H^2_n(\cle) \raro \clh
\]
is a partial isometry with the required properties. The reverse implication obviously holds. \qed

\begin{Remark}\label{canonical-sub}
Let $\cls \subset \clh$ be a closed invariant subspace of a pure commuting row contraction
$T \in \clb(\clh)^n$ and let $\Pi_{\cls} : \cls \raro H^2_n(\cle)$ be a minimal dilation of the
restriction $T|_{\cls}$. Then the map
\[
\Pi: H^2_n(\cle) \stackrel{\Pi^*_{\cls}}{\longrightarrow} \cls \hookrightarrow \clh
\]
is a partial isometry with {\rm ran}$\Pi = \cls$ such that $\Pi$ intertwines $M_z$
on $H^2_n(\cle)$ and $T$ on $\clh$ componentwise. Any map $\Pi$ arising in this way
will be called a {\rm canonical representation} of $\cls$. Note that in this situation
$\Pi^*|_{\cls} = \Pi_{\cls}$ is a minimal dilation of $T|_{\cls}$.

If $\Pi : \; H^2_n(\cle) \rightarrow \clh$ and $\tilde{\Pi} : \; H^2_n(\tilde{\cle}) \rightarrow \clh$
are two canonical representations of $\cls$, then by Theorem \ref{u-dil} there is a unitary
operator $U \in \clb(\tilde{\cle},\cle)$ such that $\tilde{\Pi} = \Pi (I_{H^2_n} \otimes U)$.
\end{Remark}

By dualizing Corollary \ref{u-dil2} one obtains the following uniqueness result.

\begin{Theorem}\label{comp-ca}
Let $T \in \clb(\clh)^n$ be a pure commuting contractive tuple and let
$\Pi : \; H^2_n(\cle) \rightarrow \clh$ be a canonical representation of a
closed $T$-invariant subspace $\cls \subset \clh$. If $\tilde{\Pi} : \; H^2_n(\tilde{\cle}) \rightarrow \clh$
is a partial isometry with $\cls = \tilde{\Pi} H^2_n(\tilde{\cle})$ and $\tilde{\Pi}
M_{z_i} = T_i \tilde{\Pi}$ for $i = 1, \ldots, n$, then there
exists an isometry $V : \cle \raro \tilde{\cle}$ such that
\[
\tilde{\Pi} = \Pi (I_{H^2_n} \otimes V^*).
\]
\end{Theorem}

\NI \textsf{Proof.} Since $\tilde{\Pi}^*$ is a partial isometry with
$\ker\tilde{\Pi}^* = ({\rm ran} \tilde{\Pi})^{\bot} = \cls^{\bot}$,
the map $\tilde{\Pi}^*: \cls \raro H^2_n(\tilde{\cle})$ is an isometry. The
adjoint of this isometry intertwines the tuples $M_z \in \clb(H^2_n(\tilde{\cle}))^n$
and $T|_{\cls}$. Hence $\tilde{\Pi}^*: \cls \raro H^2_n(\tilde{\cle})$ is a
dilation of $T|_{\cls}$. Since $\Pi^*: \cls \raro H^2_n(\cle)$ is a minimal
dilation of $T|_{\cls}$, Corollary \ref{u-dil2} implies that there is an isometry
$V: \cle \raro \tilde{\cle}$ such that
\[
\tilde{\Pi}^*|_{\cls} = (I_{H^2_n} \otimes V) \Pi^*|_{\cls}.
\]
By taking adjoints and using the fact that {\rm ran} $\Pi =$ {\rm ran} $\tilde{\Pi} = \cls$,
we obtain that
\[
\tilde{\Pi} = \Pi (I_{H^2_n} \otimes V^*).
\]
Thus the proof is complete.
\qed

\begin{Corollary}\label{comp-aa}
Let $T \in \clb(\clh)^n$ be a pure commuting contractive tuple and let $\cls \subset \clh$
be a closed $T$-invariant subspace. Suppose that
\[
\Pi_j : H^2_n(\cle_j) \raro \clh \quad (j = 1,2)
\]
are partial isometries with range $\cls$ such that $\Pi_j$ intertwines $M_z$ on
$H^2_n(\cle_j)$ and $T$ on $\clh$ for $j = 1,2$. Then there exists a partial isometry $V
: \cle_1 \raro \cle_2$ such that \[\Pi_1 = \Pi_2 (I_{H^2_n} \otimes
V).\]
\end{Corollary}
\NI\textsf{Proof.} Let $\cls = \Pi H^2_n(\cle)$ be a canonical
representation of $\cls$. Theorem \ref{comp-ca} implies that
\[
\Pi_j = \Pi (I_{H^2_n} \otimes V_j^*)
\]
for some isometry $V_j : \cle
\raro \cle_j$, $j = 1,2$. Therefore,
\[
\begin{split}\Pi_1 & = \Pi
(I_{H^2_n} \otimes V_1^*) \\ & = (\Pi_2 (I_{H^2_n} \otimes V_2))
(I_{H^2_n} \otimes V_1^*) \\ & = \Pi_2 (I_{H^2_n} \otimes V_2 V_1^*)
\\ & = \Pi_2 (I_{H^2_n} \otimes V),
\end{split}
\]
where $V = V_2 V_1^*: \cle_1 \raro \cle_2$ is a partial isometry. This completes the proof.
\qed
\vspace{.6cm}

\section{Parameterizations of Wandering Subspaces}

In this section we consider parameterizations of wandering subspaces for pure
commuting row contractions $T \in \clb(\clh)^n$ which in the  case of the
one-variable shift $T = M_z \in  \clb(H^2_{\cle}(\mathbb D))$ on the Hardy space
of the unit disc reduce to the representation (\ref{paramet}).
As suggested by Theorem \ref{inv}
the isometric intertwiner $V: H^2_{\clw}(\mathbb D) \rightarrow H^2_{\cle}(\mathbb D)$ from
the introduction is replaced by a partial isometry $\Pi: H^2_n(\cle) \rightarrow H$
intertwining $M_z \in \clb(H^2_n(\cle))^n$ and $T \in \clb(\clh)^n$.

We begin with an elementary but crucial observation concerning the uniqueness of
wandering subspaces for commuting tuples of operators.

Let $\clw$ be a wandering subspace for a commuting tuple $T \in \clb(\clh)^n$. Set
\[
\clg_{T, \clw} = \bigvee_{\bm{k} \in \mathbb N^n} T^{\bm{k}} \clw .
\]
An elementary argument shows that
\[
\clg_{T, \clw} \ominus \sum_{i=1}^n T_i \clg_{T, \clw} =
\bigvee_{\bm{k} \in \mathbb{N}^n} T^{\bm{k}} \clw \ominus \bigvee_{\bm{k} \in
\mathbb{N}^n \setminus \{\bm{0}\}} T^{\bm{k}} \clw = \clw.
\]
It follows that
\[
\clw = \mathop{\cap}_{i=1}^n (\clg_{T, \clw} \ominus T_i \clg_{T, \clw}).
\]
Consequently, we have the following result.

\begin{Proposition}\label{u-wandering}
Let $T \in \clb(\clh)^n$ be a commuting tuple of bounded operators on a Hilbert space $\clh$ and let $\clw$ be a
wandering subspace for $T$. Then
\[\clw = {\cap}_{i=1}^n (\clg_{T, \clw} \ominus T_i \clg_{T, \clw})
= \clg_{T, \clw} \ominus \sum_{i=1}^n T_i \clg_{T, \clw}.\] In
particular, if $\clw$ is a generating wandering subspace for $T$,
then \[\clw = \cap_{i=1}^n (\clh \ominus T_i \clh) = \clh \ominus
\sum_{i=1}^n T_i \clh.\]
\end{Proposition}

Starting point for our description of wandering subspaces is the following
general observation.

\begin{Theorem}\label{inv-tuple}
Let $T \in \clb(\clh)^n$ be a commuting tuple of bounded operators on a
Hilbert space $\clh$ and let $\Pi: H^2_n(\cle) \raro \clh$ be a partial isometry
with $\Pi M_{z_i} = T_i \Pi$ for $i = 1, \ldots, n$. Then $\mathcal{S} = \Pi (H^2_n(\cle))$
is a closed $T$-invariant subspace, $\clw = \cls \ominus \sum_{i=1}^n T_i \cls$ is a wandering
subspace for $T|_{\cls}$ and
\[
\clw = \Pi ((\ker \Pi)^{\bot} \cap \cle).
\]
\end{Theorem}

\begin{proof}
Define $\clf = (\ker \Pi)^{\bot} \cap \cle .$
Obviously the range $\cls$ of $\Pi$ is a closed $T$-invariant subspace and
$\clw = \cls \ominus \sum_{i=1}^n T_i \cls$ is a wandering subspace for $T|_{\cls}$.
To prove the claimed representation of $\clw$ note first that
\[
\clw = \Pi (H^2_n(\cle)) \ominus \sum_{i=1}^n T_i \Pi (H^2_n(\cle))
= \Pi (H^2_n(\cle)) \ominus \sum_{i=1}^n \Pi M_{z_i} (H^2_n(\cle)).
\]
For $f \in \clf, h \in H^2_n(\cle)$ and $i = 1, \ldots ,n$, we have
\[
\langle \Pi f, \Pi z_i h \rangle = \langle \Pi^* \Pi f, z_i h \rangle = \langle f, z_i h \rangle = 0.
\]
Conversely, each element in $\clw$ can be written as $\Pi f$ with $f \in (\ker \Pi)^{\bot}.$
But then, for $h \in H^2_n(\cle)$ and $i = 1, \ldots ,n$, we obtain
\[
\langle f, z_i h \rangle = \langle \Pi f, \Pi z_i h \rangle = 0.
\]
To conclude the proof it suffices to recall that
$H^2_n(\cle) \ominus \sum_{i=1}^n M_{z_i} (H^2_n(\cle)) = \cle$.
This identity is well known, but also follows directly from Proposition \ref{u-wandering},
since $\cle$ is a generating wandering subspace for $M_z \in \clb(H^2_n(\cle))^n.$
\end{proof}

If $T \in \clb(\clh)^n$ is a pure commuting row contraction, then by Theorem \ref{inv} each
closed $T$-invariant subspace $\cls \subset \clh$ admits a representation $\mathcal{S} = \Pi (H^2_n(\cle))$
as in the hypothesis of the preceding theorem.

\begin{Corollary}\label{generated_space}
In the setting of Theorem \ref{inv-tuple} the identity
\[
\bigvee_{\bm{k} \in \mathbb N^n} T^{\bm{k}} \clw = \overline{\Pi H^2_n(\clf)}
\]
holds with $\clf = (\ker \Pi)^{\bot} \cap \cle$.
\end{Corollary}

\textsf{Proof.} \NI  Since $\clw = \Pi \clf$, we have
\[
\bigvee_{\bm{k} \in \mathbb{N}^n} T^{\bm{k}} \clw = \bigvee_{\bm{k} \in \mathbb{N}^n} T^{\bm{k}} \Pi \clf  =
\bigvee_{\bm{k} \in \mathbb{N}^n} \Pi M_z^{\bm{k}} \clf = \overline{ \Pi (\mathop{\vee}_{\bm{k} \in \mathbb{N}^n}
M_z^{\bm{k}} \clf)} = \overline{\Pi H^2_n(\clf)}.
\]
This completes the proof. \qed

\section{Contractive analytic Hilbert spaces and inner functions}

Let $K : \mathbb{B}^n \times \mathbb{B}^n \raro \mathbb{C}$ be a
positive definite function and let $\clh_K$ be the functional Hilbert
space with reproducing kernel $K$. We say that $\clh_K$
is a \textit{contractive analytic Hilbert space} (cf. \cite{JS1}, \cite{JS2})
over $\mathbb{B}^n$ if $\clh_K$ consists of analytic functions on $\mathbb B^n$
and if the multiplication tuple $ M_z = (M_{z_1}, \ldots, M_{z_n})$ is a pure row
contraction on $\clh_K$.

\NI Typical and important examples of contractive analytic Hilbert spaces
include the Drury-Arveson space, the Hardy space and the weighted
Bergman spaces over $\mathbb{B}^n$ (cf. Proposition 4.1 in
\cite{JS2}).

Let $\clh_{K}$ be a contractive analytic Hilbert space and let $\cle$ and
$\cle_*$ be arbitrary Hilbert spaces. An operator-valued map
$\Theta : \mathbb{B}^n \raro \clb(\cle, \cle_*)$ is said to be a
\textit{$K$-multiplier} if
\[
\Theta f \in \clh_{K} \otimes \cle_*
\quad {\rm for \; every} \; f \in H^2_n \otimes \cle.
\]
The set of all
$K$-multipliers is denoted by $\clm(H^2_n \otimes \cle, \clh_K
\otimes \cle_*)$. If $\Theta \in \clm(H^2_n \otimes \cle,
\clh_K \otimes \cle_*)$, then the multiplication operator
$M_{\Theta} : H^2_n \otimes \cle \raro \clh_K \otimes
\cle_*$ defined by
\[
(M_{\Theta} f)(\w) = (\Theta f)(\w) = \Theta(\w) f(\w)
\]
is bounded by the closed graph theorem. We shall call a multiplier
$\Theta$ partially isometric or isometric if the induced multiplication operator
$M_{\Theta}$ has the corresponding property.

The space of $K$-multipliers can be described in the following way
(cf. Corollary 4.3 in \cite{JS2}). Let $X$ be in $\clb(H^2_n \otimes
\cle, \clh_{K} \otimes \cle_*)$. Then $X = M_{\Theta}$ for some $\Theta \in \clm(H^2_n \otimes
\cle, \clh_{K} \otimes \cle_*)$ if and only if
\[
X(M_{z_i} \otimes I_{\cle}) = (M_{z_i} \otimes I_{\cle_*}) X\quad {\rm for} \; i = 1, \ldots, n.
\]

\begin{Definition}\label{K-inner}
{\rm Let $\Theta : \mathbb{B}^n \raro \clb(\cle, \cle_*)$ be an operator-valued  function and
let $\clh_K$ be a contractive analytic Hilbert space as above. Then $\Theta$ is said to be a}
\it{$K$-inner function} {\rm if $\Theta x \in  \clh_K
\otimes \cle_*$ with
$\|\Theta x \|_{\clh_K \otimes \cle_*} = \| x \|_{\cle}$ for all $ x \in \cle$ and if}
\[
\Theta \cle \perp M_z^{\bm{k}} (\Theta \cle) \quad {\rm for \; all} \; \bm{k} \in \mathbb{N}^n
\setminus \{0\}.
\]
\end{Definition}

The notion of $K$-inner functions for the particular case of weighted Bergman spaces
on $\mathbb{D}$ is due to A. Olofsson
\cite{O2}. His definition of Bergman inner functions was
motivated by earlier observations \cite{HH} of H. Hedenmalm concerning invariant subspaces and
wandering subspaces of the Bergman space on $\mathbb{D}$.

The following result should be compared with Theorem 6.1 in \cite{O2} or Theorem 3.3 in
\cite{HKZ}.

\begin{Theorem}\label{inner-multiplier}
Each $K$-inner function $\Theta : \mathbb{B}^n \raro \clb(\cle, \cle_*)$ is a
contractive $K$-multiplier.
\end{Theorem}

\NI \textsf{Proof.} Let $\Theta : \mathbb{B}^n \raro \clb(\cle, \cle_*)$ be a
$K$-inner function. Then $\clw = \Theta \cle \subset \clh_K \otimes \cle_*$ is a
generating wandering subspace for the restriction of $M_z \in \clb(\clh_K \otimes \cle_*)^n$
to the closed invariant subspace
\[
\cls = \bigvee_{\bm{k} \in \mathbb N^n} (M_z^{\bm{k}} \clw) \subset \clh_K \otimes \cle_*.
\]
By Theorem \ref{inv} and the remarks preceding Definition \ref{K-inner}, there is a partially
isometric $K$-multiplier $ \hat{\Theta} \in \clm(H^2_n \otimes \hat{\cle}, \clh_K \otimes \cle_*)$
such that $\cls$ is the range of the induced multiplication operator $M_{\hat{\Theta}}$. Define
$ \clf = (\ker M_{\hat{\Theta}})^{\perp} \cap \hat{\cle}$. A
straightforward application of Proposition \ref{u-wandering} and Theorem \ref{inv-tuple}
yields that
\[
M_{\hat{\Theta}}: \clf \raro \clw
\]
is a unitary operator. Since also $M_{\Theta}: \cle \raro \clw$ is a unitary operator, it follows
that there is a unitary operator $U: \cle \raro \clf$ such that $\Theta(z) = \hat{\Theta}(z)U$ for
all $\z \in \mathbb B^n$.
For each function $f \in H^2_n \otimes \cle \subset \clo(\mathbb B^n,\cle)$, it follows that
\[
\Theta(\z)f(\w) = \hat{\Theta}(\z)Uf(\w)
\]
for all $\z,\w \in \mathbb B^n$.  Evaluating this identity for $\z = \w$, we obtain that
$\Theta f = \hat{\Theta}Uf$ for all $f \in H^2_n(\cle)$. Since $H^2_n(\cle) \rightarrow H^2_n(\hat{\cle})$,
$f \mapsto Uf$, is isometric and since $\hat{\Theta}  \in \clm(H^2_n \otimes \hat{\cle}, \clh_K \otimes \cle_*)$
 is a contractive $K$-multiplier, it follows that also
$\Theta$ is a contractive $K$-multiplier.
\qed

In the scalar case $\cle = \cle_* = \mathbb C$ the preceding theorem implies that each $K$-inner function
$\Theta : \mathbb B^n \rightarrow \clb(\mathbb C) \cong \mathbb C$ satisfies the estimates
\[
1 = \| \Theta \|_{\clh_K} \leq \| \Theta \|_{\clm(H^2_n,\clh_K)} \| 1 \|_{H^2_n} \leq 1.
\]
Hence the norm of $\Theta$, and also of each scalar multiple of $\Theta$, as an element in $\clh_K$ coincides with
its norm as a multiplier from $H^2_n$ to $\clh_K$. We apply this observation to a natural class of examples.

A polynomial $p = \sum_{\bm{k}\in\mathbb N^n} a_{\bm{k}} z^{\bm{k}}$ is called {\it quasi-homogeneous} if there
are a tuple $\bm{m} = (m_1, \ldots ,m_n)$ of positive integers $m_i$ and an integer $\ell \geq 0$ such that
$\sum^n_{i=1} m_i k_i = \ell$ for all $\bm{k} \in \mathbb N^n$ with $a_{\bm{k}} \neq 0$. In this case
$p$ is said to be $\bm{m}$-quasi-homogeneous of degree $\ell$. Let us denote by $R_{\bm{m}}(\ell)$ the
set of all $\bm{m}$-quasi-homogeneous polynomials of degree $\ell$.

\begin{Corollary}\label{quasi-homogeneous}
Suppose that $\clh_K$ is a contractive analytic Hilbert space on $\mathbb B^n$ such that the
monomials $z^{\bm{k}}$ $(\bm{k} \in \mathbb N^n)$ form an orthogonal basis of $\clh_K$. Let
$p \in \mathbb C[z_1, \ldots , z_n]$ be a quasi-homogeneous polynomial. Then
\[
\| p \|_{\clh_K} = \| p \|_{\clm(H^2_n,\clh_K)}.
\]
If $\| p \|_{\clh_K} = 1$, then $p$ is a $K$-inner function.
\end{Corollary} 

\NI \textsf{Proof.} Suppose that $p \in R_{\bm{m}}(\ell)$ is $\bm{m}$-quasi-homogeneous of degree $\ell$
with $\| p \|_{\clh_K} = 1$. Then $z^{\bm{k}} p $ is  $\bm{m}$-quasi-homogeneous of degree 
$\ell + \sum^n_{i=1} m_i k_i$ for $\bm{k} \in \mathbb N^n$. Since by hypothesis
\[
\clh_K = \oplus_{\ell} R_{\bm{m}}(\ell),
\]
it follows that $p$ is a $K$-inner function. The remaining assertions follow from the remarks
following Theorem \ref{inner-multiplier}.
\qed

If $\Theta : \mathbb{B}^n \raro \clb(\cle, \cle_*)$ is a $K$-inner function, then
$\clw = \Theta \cle \subset \clh_K \otimes \cle_*$ is a closed subspace which is the
generating wandering subspace for $M_z$ restricted to $\cls = \bigvee_{\bm{k} \in \mathbb N^n}M_z^{\bm{k}}\clw$.
Hence in the setting of Corollary \ref{quasi-homogeneous} each closed $M_z$-invariant 
subspace $\cls = \bigvee_{\bm{k} \in \mathbb N^n}M_z^{\bm{k}}p \subset \clh_K$ generated by a 
quasi-homogeneous polynomial $p$ is generated by the wandering subspace $\clw = \mathbb C p = 
\cls \ominus \sum^n_{i=1} M_{z_i} \cls$.

Corollary \ref{quasi-homogeneous} applies in particular to the functional Hilbert spaces $H(K_{\lambda})$
$(\lambda \geq 1)$ with reproducing kernel $K_{\lambda}(\bm{z},\bm{w}) = 
(1 - \sum^n_{i=1} z_i\overline{w}_i)^{-\lambda}$, since these spaces satisfy all hypotheses for
$\clh_K$ contained in Corollary \ref{quasi-homogeneous}.

\begin{Remark}\label{K-multipliers}
{\rm It is well known (cf. Corollary 3.3 in \cite{W}) that the unit sphere $\partial \mathbb B^n$ is
contained in the approximate point spectrum $\sigma_{\pi}(M_z,H^2_n)$ of the multiplication
tuple $M_z \in \clb(H^2_n)^n$. Since the approximate point spectrum satisfies an analytic
spectral mapping theorem (see Section 2.6 in \cite{EP} for the relevant definitions
and the spectral mapping theorem), in dimension $n > 1$}
\[
0 \in p(\sigma_{\pi}(M_z,H^2_n)) = \sigma_{\pi}(M_p,H^2_n)
\]
{\rm for each homogeneous polynomial $p$ of positive degree. Hence, for each such polynomial, the
subspace $pH^2_n \subset H^2_n$ is non-closed. An elementary argument, using the fact that the
inclusion $H^2_n \subset H(K_{\lambda})$ is continuous, shows that also $pH^2_n \subset H(K_{\lambda})$
is not closed. It follows that in dimension $n > 1$ there is no chance to show that
$K_{\lambda}$-inner multipliers have the expansive multiplier property proved in \cite{O3} for
operator-valued Bergman inner functions on the unit disc.}
\end{Remark}

By combining Theorem \ref{inv} (or Theorem 3.2 in \cite{JS2}), Theorem
\ref{comp-ca} and Corollary \ref{comp-aa}
we obtain the following characterization of invariant subspaces of
vector-valued contractive analytic Hilbert spaces.

\begin{Theorem}\label{inv-rkhs}
Let $\clh_K$ be a contractive analytic Hilbert space over $\mathbb{B}^n$ and let
$\cle_*$ be an arbitrary Hilbert space. Then a closed subspace $\cls \subset
\clh_K \otimes \cle_*$ is invariant for
$M_z \otimes I_{\cle_*}$ if and
only if there exists a Hilbert space $\cle$ and a partially
isometric $K$-multiplier $\Theta \in \clm(H^2_n \otimes \cle,
\clh_{K} \otimes \cle_*)$ with
\[
\cls = \Theta H^2_n(\cle).
\]
If $\cls = \tilde{\Theta} H^2_n(\tilde{\cle})$ is another
representation of the same type,
then there exists a partial isometry $V :\cle \raro \tilde{\cle}$
such that
\[
\Theta(\z) = \tilde{\Theta}(\z) V \quad \quad (\z \in \mathbb{B}^n).
\]
Furthermore, if $\cls = \Theta_c H^2_n(\cle_c)$ is a canonical
representation of $\cls$ in the sense of Remark \ref{canonical-sub},
then
\[
\Theta_c(\z) = \Theta(\z) V_c \quad \quad \quad (\z \in
\mathbb{B}^n)
\]
for some isometry $V_c : \cle_c \raro \cle$.
\end{Theorem}

The first part of the preceding theorem for the particular case
$\clh_K = H^2_n$ is the Drury-Arveson space is a result of
McCullough and Trent \cite{MT}, which generalizes the classical
Beurling-Lax-Halmos theorem to the multivariable case.
The last part seems to be new even in the case of
the Drury-Arveson space.

By applying Theorem \ref{inv-tuple} we obtain a generalization
of a result of Olofsson (Theorem 4.1 in \cite{O2}) to a quite
general multlivariable setting. 

\begin{Theorem}\label{wandering-rkhs1}
Let $\clh_K$ be a contractive analytic Hilbert space on $\mathbb{B}^n$
and let $\cle_*$ be an arbitrary Hilbert space.
Let $\cls = \Theta H^2_n(\cle)$ be a closed $M_z$-invariant
subspace of $\clh_K \otimes \cle_*$ represented by a partially
isometric multiplier $\Theta \in \clm(H^2_n \otimes \cle,
\clh_{K} \otimes \cle_*)$ as in Theorem \ref{inv-rkhs}. Then 
$\Theta_0 : \mathbb B^n \rightarrow \clb(\clf,\cle_*), \Theta_0(z) = \Theta(z)|_{\clf}$,
where
\[
\clf = \{\eta \in \cle: M_{\Theta}^*M_{\Theta} \eta = \eta\},
\]
defines a $K$-inner function such that the wandering subspace
$\clw = \cls \ominus \sum^n_{i=1} M_{z_i} \cls$ of $M_z$ restricted to $\cls$ is given by
\[
\clw = \Theta_0 \clf.
\]
The wandering subspace $\clw$ is generating for $M_z|_{\cls}$ if and only if
\[
\cls = \overline{\Theta H^2_n(\clf)}.
\]
\end{Theorem}

\NI \textsf{Proof.} Since $M_{\Theta}|_{(\ker M_{\Theta})^{\bot}}$ is an isometry and since
\[
\clf = (\ker M_{\Theta})^{\bot} \cap \cle,
\]
the result follows as an application of Theorem \ref{inv-tuple} and Corollary\ref{generated_space} with 
$T = M_z \in \clb(\clh_K \otimes \cle_*)^n$.
\qed

By using once again the fact that $M_{\Theta}$ is a partial isometry one
can reformulate the necessary and sufficient condition for $\clw$ to be a generating
wandering subspace for $M_z|_{\cls}$ given in Theorem \ref{wandering-rkhs1}.

\begin{Theorem}
In the setting of Theorem \ref{wandering-rkhs1} the space $\clw = \Theta \clf$ is a generating
wandering subspace for $M_z|_{\cls}$ if and only if
\[
(\ker M_{\Theta})^\perp \cap (H^2_n(\clf))^\perp = \{0\}.
\]
\end{Theorem}

\NI \textsf{Proof.} For $f \in (\ker M_{\Theta})^{\bot}$ and
$h \in H^2_n(\clf)$, the identity
\[
\langle f,h \rangle = \langle f, P_{(\ker M_{\Theta})^{\bot}} h \rangle
= \langle \Theta f, \Theta P_{(\ker M_{\Theta})^{\bot}} h \rangle
= \langle \Theta f, \Theta h \rangle
\]
holds. Using this observation one easily obtains the identity
\[
\Theta H^2_n(\cle) \ominus \overline{\Theta H^2_n(\clf)}  =
\Theta [(\ker M_{\Theta})^\perp \cap (H^2_n(\clf))^\perp].
\]
Since by Theorem \ref{wandering-rkhs1} the space $\clw = \Theta \clf$
is a generating wandering subspace for $M_z|_{\cls}$ if and only if
the space on the left-hand side is the zero space, the assertion follows. \qed

We conclude by giving an example which shows that in the multivariable setting,
even for the nicest analytic functional Hilbert spaces on $\mathbb B^n$, there
are $M_z$-invariant subspaces which do not possess a generating wandering 
subspace.

\begin{Example} {\rm For $a \in \mathbb B^n$, define $\cls_a = \{ f \in H^2_n : f(a) = 0 \}$.
For $a \neq 0$, the wandering subspace $\clw_a = \cls_a \ominus \sum^n_{i=1} M_{z_i} \cls_a$
for $M_z$ restricted to $\cls_a$ is one-dimensional (see Theorem 4.3 in \cite{GRS}). Hence, if $n > 1$, then 
the common zero sets
\[
Z(\cls_a) = \{ a \} \neq Z(\clw_a) = Z(\bigvee_{\bm{k} \in \mathbb N^n} z^{\bm{k}} \clw_a)
\]
of $\cls_a$ and the invariant subspace generated by $\clw_a$ are different. Thus,
for $n > 1$ and $a \neq 0$, the restriction of $M_z$ to $\cls_a$ has no generating wandering 
subspace. Since $z_1, \ldots ,z_n \in \cls_0 \ominus \sum^n_{i=1} M_{z_i} \cls_0$ and since
$\cls_0 = \bigvee \{ z^{\bm{k}} z_j : \bm{k} \in \mathbb N^n \; {\rm and} \; j = 1, \ldots ,n \}$,
we obtain that $\cls_0$ possesses the generating wandering subspace 
$\clw_0 = {\rm span} \{ z_1, \ldots , z_n \}$. By using Corollary 4.6 in \cite{GRS} one sees
that the above observations remain true if $H^2_n$ is replaced by the Hardy space $H^2(\mathbb B^n)$
or the Bergman space $L^2_a(\mathbb B^n)$. }

\end{Example}

\vspace{0.2in}

\NI\textit{Acknowledgement:} The research of the fourth author was
supported in part by NBHM (National Board of Higher Mathematics,
India) Research Grant NBHM/R.P.64/2014.

\end{document}